\documentclass[12pt,leqno]{article}
\usepackage{amssymb,amsmath,amsfonts,amsbsy, amsthm, xspace, latexsym, amscd, graphicx, fancyhdr,verbatim,enumitem, pinlabel}

\theoremstyle{colon}

\swapnumbers
\theoremstyle{definition}
\newtheorem{definition}{Definition}[section]
\newtheorem{remark}[definition]{Remark}

\newtheorem{example}[definition]{Example}
\newtheorem{examples}[definition]{Examples}

\theoremstyle{plain}
\newtheorem{theorem}[definition]{Theorem}

\newtheorem{lemma}[definition]{Lemma}
\newtheorem{corollary}[definition]{Corollary}

\newlist{T-enum}{enumerate}{2}
\newlist{L-enum}{enumerate}{2}
\newlist{C-enum}{enumerate}{2}
\newlist{P-enum}{enumerate}{2}  
\newlist{Pf-enum}{enumerate}{2} 
\newlist{D-enum}{enumerate}{2}
\newlist{Ex-enum}{enumerate}{2}
\newlist{Exs-enum}{enumerate}{2}
\newlist{E-enum}{enumerate}{2}
\newlist{R-enum}{enumerate}{2}
\setlist[T-enum,1]{label=(\roman*),format=\bfseries\emph,leftmargin=*,labelindent=.1\parindent}
\setlist[T-enum,2]{label=(\alph*),format=\bfseries\emph,leftmargin=*,labelindent=.1\parindent}
\setlist[L-enum,1]{label=(\roman*),format=\bfseries\emph,leftmargin=*,labelindent=.1\parindent}
\setlist[L-enum,2]{label=(\alph*),format=\bfseries\emph,leftmargin=*,labelindent=.1\parindent}
\setlist[C-enum,1]{label=(\roman*),format=\bfseries\emph,leftmargin=*,labelindent=.1\parindent}
\setlist[C-enum,2]{label=(\alph*),format=\bfseries\emph,leftmargin=*,labelindent=.1\parindent}
\setlist[P-enum,1]{label=(\roman*),format=\bfseries\emph,leftmargin=*,labelindent=.1\parindent}
\setlist[P-enum,2]{label=(\alph*),format=\bfseries\emph,leftmargin=*,labelindent=.1\parindent}
\setlist[Pf-enum,1]{label=(\roman*), leftmargin=*,labelindent=.1\parindent}
\setlist[Pf-enum,2]{label=(\alph*), leftmargin=*,labelindent=.1\parindent}
\setlist[D-enum,1]{label=\textbf{\arabic*.},leftmargin=*,labelindent=.2\parindent}
\setlist[D-enum,2]{label=\textbf{(\alph*)},leftmargin=*,labelindent=.1\parindent}
\setlist[Ex-enum,1]{label=\textbf{\arabic*.},leftmargin=*,labelindent=.15\parindent}
\setlist[Exs-enum,1]{label=\textbf{\arabic*.},leftmargin=*,labelindent=.15\parindent}
\setlist[E-enum,1]{label=\textbf{\arabic*.},leftmargin=*,labelindent=.15\parindent}
\setlist[R-enum,1]{label=\textbf{\arabic*.},leftmargin=*,labelindent=.15\parindent}
\newcommand{\bra}{\ensuremath{\langle}}
\newcommand{\ket}{\ensuremath{\rangle}}

\newcommand{\supp}{\ensuremath{{\rm{supp}}}}
\newcommand{\card}{\ensuremath{{\rm{card}\,}}}

\newcommand{\ran}{\ensuremath{{\rm{ran}}}}

\newcommand{\emb}{\ensuremath{\hookrightarrow}}

\newcommand{\ifff}{\ensuremath{{\rm{\,iff\;}}}}

\begin{document}
\title{Infinitesimals inside the Familiar Field of Complex Numbers}
\author{Todor D. Todorov\\
Emeritus of
	California Polytechnic State University\\
	San Luis Obispo, California 93407, USA\\
	ttodorov@calpoly.edu}
	\maketitle

\begin{abstract} \emph{Abstract}: We show that the field of complex numbers $\mathbb C$ contains non-zero infinitesimals by observing that $\mathbb C$ contains non-Archimedean subfields. Our observation is based on an old  theorem in algebra due to E. Steinitz, discussed in the article in detail. The presence of infinitesimals in $\mathbb C$ was surprise to the author and might be surprise to the readers as well, since $\mathbb C$ is commonly defined in terms of the field of reals $\mathbb R$, which is Archimedean. An additional intrigue arises from the fact that $\mathbb R$ was historically  introduced in 19-th century (by Dedekind, Cauchy and others) exactly to make infinitesimals in Leibniz-Newton infinitesimal calculus redundant. It seems that mathematics will never get rid of infinitesimals completely - they are all around us whether we like it or not. In the last section of the article we explain how our result fits to analysis, both standard and non-standard. With examples from history of calculus as well of  first-class recent achievements in analysis we try to convince the reader that presence of infinitesimals in analysis simplifies its formal language and improves its efficiency. The motivations of our research is related to an attempt to simplify the properties of particular algebras of generalized functions of Colombeau's type, shortly discussed in the text.
\end{abstract}


\emph{2020 Mathematis Subject Classification}: Primary: 46F30; Secondary:  26E35; 12D15; 12F20; 12L10; 13F25; 01A50.

\section{Introduction}\label{S: Introduction}

	\quad\;We show that the field of complex numbers $\mathbb C$ contains a non-Archimedean subfield $\mathcal R$ (Section~\ref{S: Infinitesimals in Algebra}). Thus the non-zero infinitesimals in $\mathcal R$ are also elements of $\mathbb C$. The field $\mathcal R$ determines an absolute value on $\mathbb C$, which is an (proper) extension of the usual one in the cases $\mathcal R$ contains a copy of $\mathbb R$. We discuss our result from purely algebraic point of view as well from the point of view of analysis (Section~\ref{S: Infinitesimals in Analysis}).

	The set-theoretical framework of our article (Section~\ref{S: Set-Theoretical Framework}) is the most common and popular axioms in set theory:  ZFC (Zermelo-Frankel Axioms with Axiom of Choice). We sometimes also involve (mostly for the sake of simplicity)  GCH (Generalized Continuum Hypothesis).
	
	
	The existence of embedding of $\mathcal R$ into $\mathbb C$ is a consequence of an old algebraic theorem due to E. Steinitz (Section~\ref{S: Steinitz Theorem}).  A reader who is familiar with Steinitz theorem might skip this section and jump to the next one. Actually, we show that $\mathbb C$ contains not one, but many, potentially infinitely many (mutually non-isomorphic), non-Archimedean fields $\mathcal R$ and we present a list of some of them (Examples~\ref{Exs: Non-Archimedean Real Closed Fields of Cardinality c}).
In Section~\ref{S: Uniqueness up to Isomorphism} we shortly discuss the isomorphism between some of these fields. In Section~\ref{S: kappa-Complex Numbers} we generalize the result of Section~\ref{S: Infinitesimals in Algebra} and show that not only $\mathbb C$, but all fields of $\kappa$-complex numbers also contains infinitesimals.  We also present a generalized version of fundamental theorem of algebra involving algebraically closed fields with arbitrarily large uncountable cardinality (Theorem~\ref{T: Generalized Fundamental Theorem of Algebra}). 

	In Section~\ref{S: Infinitesimals in Analysis} we discuss infinitesimals in $\mathbb C$ from the point of view of analysis, both standard and non-standard. We also mention  some more recent achievements of non-standard analysis, some of which, at least for now, remain without  standard counterpart (Remark~\ref{R: Standard vs. Non-Standard  Analysis}). Actually, one of the motivations of our research is related to an attempt to simplify the properties of particular algebras of generalized functions of Colombeau~\cite{jCol84a} type, shortly discussed at the end of (Examples~\ref{Exs: Non-Archimedean Real Closed Fields of Cardinality c}) and (Remark~\ref{R: Standard vs. Non-Standard  Analysis}).
	
	We have made efforts to make our text accessible for mathematicians with expertise in different branches of mathematics. For the purpose we present in Sections~\ref{S: Preliminaries: Ordered Fields}-\ref{S: Steinitz Theorem}  a generous list of definitions and facts (in a \emph{index-references} style) belonging to different branches of pure mathematics supplied with relevant citations. Many readers might decide to skip reading at least one of these sections (or browse quickly through). Also, some of the examples in 
Examples~\ref{Exs: Non-Archimedean Real Closed Fields of Cardinality c} have origin in non-standard analysis and we involve these examples  in our discussion in Section~\ref{S: Infinitesimals in Analysis}. The reader without background on non-standard analysis might very well skip the ``non-standard part of the text'', since our main result - infinitesimals in $\mathbb C$, does not depend on non-standard mathematics. 

	As far as we can detect, our observation, that $\mathbb C$ contains infinitesimals, is widely unknown, and it is very likely new. In particular,  the isomorphisms  $\mathbb C\cong{^*\mathbb C}\cong{^\rho\mathbb C}$ (Section~\ref{S: Infinitesimals in Algebra}) had certainly not been observed and exploited so far in mathematical research involving the fields  $\mathbb C$, ${^*\mathbb C}$ and ${^\rho\mathbb C}$ (Oberguggenberger \&Todorov~\cite{OberTod98},  Vernaeve~\cite{HVern03},  Todorov\&Vernaeve~\cite{TodVern08}, Todorov~\cite{tdTodAxioms10}-\cite{tdTodSteady2015}). 
	
	We can only speculate why mathematicians, who have been familiar with Steinitz theorem for a long time, have failed to notice the presence of infinitesimals in $\mathbb C$. Technically speaking, a such observation was possible soon after the publication of Steinitz theorem~\cite{eSteinitz1910}  in 1910, since the non-Archimedean field $\mathbb R\{t\}$ of Puiseux-Newton series was used by Newton in the distant 1736. And actually, no more is needed to make our observation. One possible explanation (shortly hinted in the abstract) is that infinitesimals have been out of fashion in modern mathematics for a long time. A student majoring mathematics nowadays might very well have never heard about the concept of infinitesimal. Why? The reason is rooted in history of calculus. In attempt to explain this, let us recall  how Leibniz derived, say, the formula $\frac{d}{dx}(x^3)=3x^2$. By ``definition'' $\frac{d}{dx}(x^3)$ is what the following algorithm produces:
\begin{description}
\item \emph{Step 1:} Let $x$ be a ``fixed ordinary quantity'' (a fixed real number in modern language) and $dx\not=0$ be ``non-zero quantity'' (a variable taking non-zero values in modern language). We calculate:
\[
\frac{(x+dx)^3-x^3}{dx}=\frac{x^3+3x^2dx+3xdx^2+dx^3-x^3}{dx}=3x^2+3xdx+dx^2.
\]
\item \emph{Step 2:} Let $dx=0$ and the result is: $
\frac{d}{dx}(x^3)=3x^2$.
\end{description}


Notice that modern computers nowadays calculate symbolically derivatives of real functions following similar algorithms  (Shamseddine, Berz~\cite{kShamseddine2010}). From the point of view of pure mathematics, however, there is clearly  \emph{lack of consistency} between Step 1 ($dx\not=0$) and Step 2 ($dx=0$), which might be phrased even as a \emph{infinitesimal paradox}: ``$dx\not=0$ and $dx=0$''. To avoid this paradox, Leibniz (Newton, Euler and others) introduced the concept of non-zero infinitesimals; $dx$ must be a non-zero infinitesimal, which allows dropping the term $3xdx+dx^2$, while keeping $3x^2$ on the right hand side of the above formula. Persistent doubts however, in the existence of non-zero infinitesimals (as a tool to avoid the above paradox), led historically to a  \emph{movement against infinitesimals, which gradually resulted to a purge of infinitesimals from mathematics} (Hall \&Todorov~\cite{HallTodLeibniz11}). The movement  started in 19-th century by Bolzano, Weirstrass and others and it is still lingering today in spite of the advent of non-standard analysis (Robinson~\cite{aRob66}) at the middle of 20-th century. For those interested in the pre-history of infinitesimals in the context of Reformation, Counter-Reformation and English Civil War, we recommend the excellent book by Alexander~\cite{aAlexander2014}.

	While dealing with infinitesimals in $\mathbb C$ we are facing the following \emph{philosophical alternative}:
\begin{enumerate}
\item Infinitesimals exist in $\mathbb C$ regardless whether we realized it or not - very much like  microbes swimming in a drop of water -  invisible for the naked eye. Perhaps, the absolute value $|\cdot|_{\mathcal R}$ (with non-Archimedean field $\mathcal R$) is our ``microscope'' under which the infinitesimal become ``visible"? 
\textbf{\large{or}}
\item $\mathbb C$ is free of infinitesmals (as it should be). We (humans) create infinitesimals in $\mathbb C$ by artificially embedding a non-Archimedean fields into $\mathbb C$?
\end{enumerate}

	As we know too well, mathematics is not equipped to answer any philosophical questions and we are not going to do that. Still, we somehow prefer to believe that infinitesimals do, indeed, exist in $\mathbb C$  (like microbes in a drop of water) and become observable  only for those who want to observe them. 

\section{Set-Theoretical Framework}\label{S: Set-Theoretical Framework}

\quad\; Our set-theoretical framework is the usual ZFC (Zermelo-Frankel Axioms with Axiom of Choice) along with the GCH (Generalized Continuum Hypothesis) in the form of $2^\kappa=\kappa_+$ for all cardinal numbers $\kappa$ (or equivalently, $2^{\aleph_\alpha}=\aleph_{\alpha+1}$ for all ordinals $\alpha$). Here $\kappa_+$ stands for the successor of $\kappa$. In particular case of $\kappa=\aleph_0$,  GCH reduces to CH (Continuum Hypothesis), which can be written in the form $\frak c=\aleph_1$, since $2^{\aleph_0}=\frak c$. Here $\aleph_0=\card(\mathbb N)$ and $\frak c=\card(\mathbb R)$. Due to the assumption of axiom of choice (AC) we have the formula $\kappa_1+\kappa_2 =\kappa_1\kappa_2=\max\{\kappa_1,\, \kappa_2\}$ for every two cardinal numbers $\kappa_1$ and $\kappa_2$, assuming in addition that at least one of them is infinite.

	Actually, we do not need GCH (neither CH) for our main observation - that $\mathbb C$ contains non-zero infinitesimals. We impose GCH mostly for the sake of convenience - to be able to establish an isomorphism between some of the fields in Examples~\ref{Exs: Non-Archimedean Real Closed Fields of Cardinality c} and to calculate more easily their cardinalities. However, we try to localize the dependence (if any) of our results on GCH or CH (or even on AC) whenever possible. Otherwise, no particular  expertise in set theory is expected from the reader.
\section{Preliminaries: Ordered Fields}\label{S: Preliminaries: Ordered Fields}

	\quad\; For convenience of the reader we present a list of basic definitions and facts about totally ordered fields. For more detailed expositions we refer to Lang~\cite{sLang},\, Hunderford~\cite{tHungerford} and  Waerden ~\cite{VanDerWaerden}. A reader, who is familiar with the topic, might jump directly to the next section. 
\begin{enumerate}
\item As usual, we denote $\mathbb N$, $\mathbb Q$, $\mathbb R$ and $\mathbb C=\mathbb R(i)$ the sets of natural, rational, real and complex numbers, respectively, and we also use $\mathbb N_0=\mathbb N\cup\{0\}$. If $S$ is a set, we denote by $\mathcal P(S)$ the power set of $S$.
\item If $\mathbb K$ is a totally ordered field (an ordered field) and $x\in\mathbb K$, we let $|x|=\max\{-x, x\}$, where $\max$ is relative to the order on $\mathbb K$. We say that
\begin{enumerate}
\item $x$ is \emph{finite} if $|x|\leq n$ for some $n\in\mathbb N$.
\item $x$ is \emph{infinitesimal} if $|x|< 1/n$ for all $n\in\mathbb N$.
\item $x$ is \emph{infinitely large} if $n<|x|$ for all $n\in\mathbb N$.
\end{enumerate}
Recall that every totally ordered field contains a copy of the field of rational numbers $\mathbb Q$. Thus both $n$ and $1/n$ are in $\mathbb K$ (justifying the above definition).
\item We denote by $\mathcal F$, $\mathcal I$ and $\mathcal L$ the set of the finite, infinitesimal and infinitely large numbers in $\mathbb K$, respectively. We also denote by $\mathcal U=\mathcal F\setminus\mathcal I$ the group of units of the ring $\mathcal F$. We observe that  $\mathcal F\cup\mathcal L=\mathbb K$,  $\mathcal F\cap\mathcal L=\varnothing$ and $\mathcal I\subset\mathcal F$. Moreover, if $x\in\mathbb K\setminus\{0\}$, then $x\in\mathcal I$ \ifff $1/x\in\mathcal L$. We sometimes write $x\approx y$ instead of $x-y\in\mathcal I$.
\item A totally ordered field $\mathbb K$ is called\, \emph{Archimedean} if  $\mathcal L=\varnothing$. The latter is equivalent to either $\mathcal I=\{0\}$,  or $\mathcal F=\mathbb K$.   Consequently, $\mathbb K$ is \emph{non-Archimedean} \ifff either of the following holds: $\mathcal L\not=\varnothing$, $\mathcal I\not=\{0\}$, or  $\mathcal F\not=\mathbb K$.
\item A totally ordered field $\mathbb K$ is \emph{algebraically saturated} if every nested sequence of open intervals in $\mathbb K$ has a non-empty intersection. More generally, $\mathbb K$ is \emph{algebraically $\kappa_+$-saturated} for some infinite cardinal $\kappa$ if every family of open intervals in $\mathbb K$ of cardinality $\kappa$ with the finite intersection property has a non-empty intersection. Moreover, $\mathbb K$ is \emph{fully saturated} if it is $\kappa_+$-saturated for ${\kappa_+}=\card(\mathbb K)$.
\item In addition to what was mentioned above, a totally ordered field $\mathbb K$ is\, \emph{Archimedean} if and only if $\mathbb K$ is an ordered subfield (can be embedded as such) of the field of real numbers $\mathbb R$. Consequently, every totally ordered field $\mathbb K$ which contains a proper copy of $\mathbb R$ (as an ordered subfield) is \emph{non-Archimedean}.  On the other hand, it is clear that an algebraically saturated ordered field must be non-Archimedean. Perhaps, the simplest example of a non-saturated non-Archimedean field, which does not contain a copy of $\mathbb R$, is given by the field $\mathbb Q(t)=\frac{\mathbb Q[t]}{\mathbb Q[t]}$ of rational functions in one variable with rational coefficients. First, we embed $\mathbb Q$ into $\mathbb Q(t)$ by $q\mapsto q\,t^0$ and convert $\mathbb Q(t)$ into a totally ordered field by ordering the ring of polynomials $\mathbb Q[t]$: A non-zero polynomial $P\in\mathbb Q[t]$ is positive if the coefficient in front of the lowest power of $t$ in $P$ is positive. For example, $2t^2-100t^3>0$, because $2>0$. We observe that $t,\,  t^2,\; 2t^2-100t^3,\; ,\dots$ are \emph{positive infinitesimals}. In particular, $t$ is a positive infinitesimal, because $\frac{1}{n}-t>0$ for all $n\in\mathbb N$. Similarly,  
$t^{-1},\; t^{-2}$, $\frac{1}{2t^2-100t^3},\dots$ are \emph{positive infinitely large} elements of $\mathbb Q(t)$.
\item A field $\mathcal R$ is \emph{formally real} (or \emph{orderable}) if $x_1^2+x_2^2+\dots+x_n^2=0$ always implies $x_1=x_2=\dots=x_n=0$. The fields $\mathbb Q$, $\mathbb R$ and $\mathbb Q(t)$ are formally real, while $\mathbb C$ is not. 
\item A formally real field $\mathcal R$ is \emph{real closed} if

\begin{enumerate} 
\item For every $a\in\mathcal R$ either $a$, or $-a$ has a square root.
\item Every polynomial $P\in\mathcal R[x]$ (with coefficients in $\mathcal R$) of \emph{odd degree} has a root in $\mathcal R$. 
\end{enumerate}
\item $\mathbb R$ is a real closed field, but  $\mathbb Q$ is not. Similarly, $\mathbb Q(t)$ is not real closed. More examples of real closed fields (Archimedean and non-Archimedean alike) appear later in the article (Examples~\ref{Exs: Non-Archimedean Real Closed Fields of Cardinality c}). 
\item Every real closed field $\mathcal R$ is a \emph{totally ordered field} under the unique ordering: $x\geq 0$ in $\mathcal R$ if $x=y^2$ for some $y\in\mathcal R$. Consequently, $\mathcal R$ is a field of zero-characteristic in the sense that $\mathcal R$ contains $\mathbb Q$ as a subfield.
\item   If $\mathcal R$ is a real closed field, then $\mathcal F/\mathcal I$ is an Archimedean real closed field. Let $\mathcal A$ be a subfield of $\mathcal R$, which is a maximal subring of $\mathcal F$ (a field of representatives). The existence of $\mathcal A$ is guaranteed by AC. Then $\mathcal A$ is an Archimedean real closed field as isomorphic to $\mathcal F/\mathcal I$. We refer to $\mathcal A$ as an \emph{Archimedean part of} $\mathcal R$ and sometimes use the notation $\mathcal A(\mathcal R)$ or $\mathcal A_\mathcal R$, since $\mathcal R$ determines $\mathcal A$ uniquely up to isomorphism.  Consequently, $\mathcal A$ and $\mathcal A(i)$ are  subfields of $\mathbb R$  and $\mathbb C$, respectively. Notice that if $\mathcal R$ is Archimedean, then $\mathcal A(i)=\mathcal R(i)$, since $\mathcal A=\mathcal R$. If, however, $\mathcal R$ a non-Archimedean, then the field extension $\mathcal R:\mathcal A$ must be transcendental hence, $\mathcal A\subsetneqq\mathcal R$ and $\mathcal A(i)\subsetneqq\mathcal R(i)$. 
\item Let $\mathcal R$ be a real closed field and $\mathcal R(i)$ obtained by adjoining $i$ to $\mathcal R$. We define $\mathcal R$-absolute value $|\cdot|_\mathcal R: \mathcal R(i)\mapsto\mathcal R(i)$ by the familiar formula $|x+iy|_\mathcal R=\sqrt{x^2+y^2}$, where $x, y\in\mathcal R$. We observe that for the range of the absolute value we have $\ran(|\cdot|_\mathcal R)=\{x\in\mathcal R: x\geq0\}$. It is clear that  $|\cdot|_\mathcal R$ possesses the usual properties of absolute value. Moreover, $|x|_\mathcal R=\max_\mathcal R\{-x, x\}$ for all $x\in\mathcal R$, where $\max_\mathcal R$ is relative to the order in $\mathcal R$. The topology on $\mathcal R(i)$ determined by $|\cdot|_\mathcal R$ coincides with the product topology on $\mathcal R(i)$ inherited from the order topology on $\mathcal R$. For more details we refer to Section~\ref{S: Infinitesimals in Analysis}.
\item  We have $|z|_\mathcal R=|z|$ for all $z\in\mathcal A(i)$, where $\mathcal A$ as an \emph{Archimedean part of} $\mathcal R$ and $|\cdot|$ stands for the usual absolute value in $\mathbb C$. Moreover, $|\cdot|_\mathcal R$ clearly reduces to the usual absolute value on $\mathbb C$ in the particular case $\mathcal R=\mathbb R$.
\item In contrast to the above, recall that a field $\mathbb K$ is \emph{algebraically closed} if every non-constant polynomial $P\in\mathbb K[x]$ (with coefficients in $\mathbb K$) has at least one root in $\mathbb K$. The field 
$\mathbb C$ is algebraically closed, but $\mathbb R$  and $\mathbb Q(t)$ are not.
\item   Let $\mathcal R$ is a formally real field. Then $\mathcal R$ is \emph{real closed} if and only if  $\mathcal R(i)$ is \emph{algebraically closed} (Artin and Schreier~\cite{ArtinSchreier27}).
 \item Let $\mathbb F$ be a subfield of $\mathbb K$ and suppose that field extension $\mathbb K : \mathbb F$ is transcendental. A subset $\mathcal B$ of $\mathbb K$ is called a \emph{transcendence basis} of $\mathbb K: \mathbb F$ if: 
(a) $\mathcal B$ is algebraically independent over $\mathbb F$; (b) The field extension
 $\mathbb K : \mathbb F(\mathcal B)$ is algebraic. The existence of $\mathcal B$ is guaranteed by Zorn lemma. The cardinality of $\mathcal B$ is called \emph{transcendence degree} of $\mathbb K\!: \!\mathbb F$ and usually denoted by ${\rm trd}(\mathbb K\!\!: \!\!\mathbb F)$, i.e.  ${\rm trd}(\mathbb K\!\!: \!\!\mathbb F)=\card\, \mathcal B$ (Hunderford~\cite{tHungerford}). Moreover, we have the formula:
\begin{equation}\label{E: card K}
\card(\mathbb K)=\max\{\aleph_0,\, \card\, \mathbb F,\, {\rm trd}(\mathbb K:\mathbb F)\}.
\end{equation}
Indeed, $\card(\mathbb K)=\card(\mathbb F(\mathcal B))$, since the field extension $\mathbb K\!:\!\mathbb F(\mathcal B)$ is algebraic by the definition of $\mathcal B$. On the other hand, $\card(\mathbb F(\mathcal B))=\max\{\aleph_0,\, \card\, \mathbb F,\, \card\, \mathcal B\}$ by the definition of $\mathbb F(\mathcal B)$, since $\mathcal B$ is algebraically independent over $\mathbb F$.
\item Let $\mathbb F$ be the \emph{prime subfield} of $\mathbb K$ (the smallest subfield of $\mathbb K$ containing the unit of $\mathbb K$) and suppose, as before, that the field extension $\mathbb K : \mathbb F$ is transcendental. In this case  the transcendence degree ${\rm trd}(\mathbb K: \mathbb F)$ is called \emph{absolute transcendence degree} of $\mathbb K$ and denoted, simply, by  ${\rm trd}(\mathbb K)$.

\end{enumerate}
\section{ Isomorphic Fields: Steinitz Theorem}\label{S: Steinitz Theorem}

	\quad\; We recall two classical theorems in algebra about isomorphisms between  algebraically closed and real closed fields. We also present several consequences from Steinitz theorem, which will lead us eventually to our main result in the next section. A reader who is familiar with Steinitz theorem might skip  this section and jump to the next one.
 \begin{theorem}[Steinitz]\label{T: Steinitz} Every two algebraically closed fields of the same characteristic and the same absolute transcendence degree are isomorphic. 
 \end{theorem}
 \begin{proof} For the original result, we refer to (Steinitz~\cite{eSteinitz1910}, p.125). For a more recent presentation we also mention (Hunderford~\cite{tHungerford}, Ch.VI, Thm.1.12, p.317).
\end{proof}
 	In some cases we can eliminate the concept of \emph{absolute transcendence degree} from the discussion.
 \begin{corollary}[Uncountable Cardinality] \label{C: Uncountable Cardinality} All algebraically closed fields of the same characteristics and the same uncountable cardinality are mutually isomorphic. 
 \end{corollary}
 \begin{proof} This follows from Theorem~\ref{T: Steinitz}, because the absolute transcendence degree coincides with the cardinality of the field for uncountable fields. Indeed, let $\mathbb K$ be an algebraically closed field with $\card\, \mathbb K>\aleph_0$ and let $\mathbb F$ stands for its prime subfield. We observe that the formula (\ref{E: card K}) at the end of Section~\ref{S: Preliminaries: Ordered Fields} reduces to $\card(\mathbb K)=\max\{\aleph_0,\,  {\rm trd}\, \mathbb K\}$,
since $\card\, \mathbb F\leq\aleph_0$ and ${\rm trd}(\mathbb K:\mathbb F)={\rm trd}\, \mathbb K$, because $\mathbb F$ is the prime subfield of $\mathbb K$. On the other hand, the above formula implies $\card(\mathbb K)={\rm trd}\, \mathbb K$ (as desired), since $\card\, \mathbb K>\aleph_0$ by assumption. 
 \end{proof}
 
 	The above corollary justifies the following ``algebraic definition'' of the field $\mathbb C$. Notice that the field of reals $\mathbb R$ participates only through its cardinality, $\frak c$; the fact that $\mathbb R$ is a Dedekind complete topological field is not even mentioned in the definition.  (In contrast, compare with Section~\ref{S: Infinitesimals in Analysis}). 
	
\begin{definition}[Algebraic Definition of $\mathbb C$]\label{D: Abstract Definition of C} $\mathbb C$ is (by definition) an algebraically closed field of zero-characteristic and cardinality $\frak c$.
\end{definition}
 \begin{corollary}[Two Fields]\label{C: Two Fields}  The fields $\mathcal R_1(i)$ and $\mathcal R_2(i)$ are isomorphic for every two real closed fields $\mathcal R_1$ and $\mathcal R_2$ such that $\card(\mathcal R_1)=\card(\mathcal R_2)>\aleph_0$. In particular, $\mathcal R(i)$ and $\mathbb C$ are isomorphic for every real closed field $\mathcal R$ (Archimedean or not) of cardinality $\frak c$. Moreover, if  $\sigma: \mathcal R(i)\mapsto\mathbb C$ is an isomorphism, then  $\sigma(q)=q$ for all $q\in\mathbb Q$ and $\sigma(i)=i$ or $\sigma(i)=-i$. Consequently, $\sigma(x+iy)=\sigma(x)\pm i\sigma(y)$ for all $x, y\in\mathcal R$, depending on $\sigma(i)=\pm i$, respectively. 
\end{corollary}
\begin{proof} An immediate consequence of Corollary~\ref{C: Uncountable Cardinality} taking into account that both $\mathcal R_1(i)$ and  $\mathcal R_2(i)$ are algebraically closed fields by Artin-Schreier theorem (Section~\ref{S: Preliminaries: Ordered Fields}) and both fields are of zero characteristic (since $\mathcal R_1$ and $\mathcal R_2$ are of zero characteristic). Moreover, $\card(\mathcal R_1(i))=\card(\mathcal R_2(i))$, because $\card(\mathcal R_n(i))=\card(\mathcal R_n)$ , $n=1, 2$. In particular, $\mathcal R(i)$ and $\mathbb C$ are isomorphic, because $\mathbb C=\mathbb R(i)$ and $\mathbb R$ is a real closed field. Finally, $\sigma(i)=\pm i$ follows from the fact that  $\sigma(i)$ is one of the two solutions of $x^2+1=0$ in $\mathbb C$. 
 \end{proof}
 
 	The next theorem is a counterpart of Steinitz theorem for real closed fields.
  \begin{theorem}[Erd\"{o}s, Gillman, Henriksen]\label{T: Erdos, Gillman and Henriksen} 
 Every two real closed fully saturated fields $\mathcal R_1$ and $\mathcal R_2$ of the same uncountable cardinality are isomorphic. 
 \end{theorem}
 \begin{proof} We refer to (Erd\"{o}s, Gillman and Henriksen~\cite{ErdGillHenr55}, Thm.2.1, p.543). We should mention that this result requires GCH (in addition to the more conventional ZFC).
\end{proof}
\section{Infinitesimals in Algebra}\label{S: Infinitesimals in Algebra}

	\quad\; Our main result below has 
the character of a \emph{mathematical observation} and it is a  straightforward consequence of what was discussed in Section~\ref{S: Steinitz Theorem}. Our framework is purely algebraic (Definition~\ref{D: Abstract Definition of C}). The discussion of $\mathbb C$ from the point of view of analysis will be postpone for Section~\ref{S: Infinitesimals in Analysis}.
\begin{theorem}[Main Result]\label{T: Main Result} $\mathbb C$ contains a non-Archimedean totally ordered subfield $\mathcal R$. Consequently, the non-zero infinitesimals in $\mathcal R$ are also elements of $\mathbb C$.  
\end{theorem}
\begin{proof} Suppose $\mathcal R$ is a non-Archimedean real closed field of $\card(\mathcal R)=\frak c$. The existence of such a field is left to Examples~\ref{Exs: Non-Archimedean Real Closed Fields of Cardinality c} in the next section. By Corollary~\ref{C: Two Fields}, the fields $\mathcal R(i)$ and $\mathbb C$ are isomorphic. Let  $\sigma: \mathcal R(i)\mapsto\mathbb C$ be any isomorphism between them. The subfield $\sigma[\mathcal R]$ of $\mathbb C$ is also real closed field. Moreover, $\sigma[\mathcal R]$ is also non-Archimedean, because  the algebraic operations in a real closed field determines uniquely the order. Indeed, let $x$ be a positive infinitesimal in $\mathcal R$, i.e. $0<x<1/n$ for all $n\in\mathbb N$. Thus $0<\sigma(x)<1/n$ for all $n\in\mathbb N$ as well, since $\sigma(1/n)=1/n$. The latter means that $\sigma(x)$ is a positive infinitesimal in $\sigma[\mathcal R]$. Clearly, $\sigma(x)\in\mathbb C$. 
\end{proof}
\begin{definition} [Infinitesimals in $\mathbb C$]\label{D: Infinitesimals in C} Let $\mathcal R$ be a non-Archimedean real closed field of cardinality $\frak c$ and $\sigma: \mathcal R(i)\mapsto\mathbb C$ be a field isomorphism. We define the absolute value $|\cdot|_{\sigma[\mathcal R]}: \mathbb C\mapsto\mathbb C$, (with range $\ran(|\cdot|_{\sigma[\mathcal R]})=\{\sigma(\mathcal R): x\in\mathcal R, x\geq0\}$) by the formula $|\sigma(x+iy|_{\sigma[\mathcal R]}=\sigma(\sqrt{x^2+y^2})$ for all $x, y\in\mathcal R$. For the sets of the finite, infinitesimal, finite, but non-infinitesimal and infinitely large elements of $\mathbb C$ we have:
\begin{align}
&\mathcal F=\{\sigma(x+iy)\in\mathbb C:\;\sqrt{x^2+y^2}\leq n \text{ for some } n\in\mathbb N\},\notag\\ 
&\mathcal I=\{\sigma(x+iy)\in\mathbb C:\; \sqrt{x^2+y^2}< 1/n \text{ for all } n\in\mathbb N\},\notag\\ 
&\mathcal U=\mathcal F\setminus\mathcal I=\{\sigma(x+iy)\in\mathbb C:\; 1/n\leq\sqrt{x^2+y^2}\leq n \text{ for some } n\in\mathbb N\},\notag\\ 
&\mathcal L=\{\sigma(x+iy)\in\mathbb C:\; n<\sqrt{x^2+y^2} \text{ for all } n\in\mathbb N\},\notag 
\end{align}
respectively, where we have taken advantage of the fact that $\sigma(q)=q$ for all $q\in\mathbb Q$.
\end{definition}
\begin{theorem}[Algebraic vs. Transcendental]\label{T: Algebraic vs. Transcendental}
Let $\mathcal R$ be a non-Archimedean real closed field of cardinality $\frak c$ (as above). Let $\mathbb A$ denote the field of the real algebraic numbers and $\mathbb A(i)$ stand for the field of all algebraic numbers. Then:
 \begin{T-enum}
\item $\mathbb A(i)\setminus\{0\}\subset\mathcal U$. Consequently, a non-zero algebraic number cannot be neither infinitesimal, nor an infinitely large element of $\mathbb C$ (whatever is the choice  of $\mathcal R$ and $\sigma$). 
\item Non-zero infinitesimals and infinitely large numbers in $\mathbb C$ are all transcendental, i.e. $\mathcal I\setminus\{0\}\subset\mathbb C\setminus\mathbb A(i)$ and  $\mathcal L\subset\mathbb C\setminus\mathbb A(i)$.
\item Let $z\in\mathbb A(i)$ and $dx\in\mathcal I$. Then $z+dx\in\mathbb A(i)$ if and only if  $dx=0$.
\item For every non-Archimedean real closed field $\mathcal R$ of cardinality $\frak c$ and every transcendental number $z\in\mathbb C$ there exists an isomorphism $\sigma_1: \mathcal R(i)\mapsto\mathbb C$ such that $\sigma_1(z)\notin\mathcal U$. Consequently, one of the two numbers in the set $\{\sigma_1(z),\; 1/\sigma_1(z)\}$ is a non-zero infinitesimal and the other is an infinitely large element of $\mathbb C$.
\end{T-enum}
\end{theorem}
\begin{proof}
\begin{Pf-enum}
\item Suppose $z\in\mathbb A(i)\setminus\{0\}$. We have $P(z)=0$ for some polynomial $P(t)=q_0t^d+q_1t^{d-1}+\dots+q_d$, where $d\in\mathbb N$, $d\geq 1$, and  $q_k\in\mathbb Q$, $k=0, 1,\dots, d$, since $\mathbb A(i)$ is an algebraically closed field. Thus $|z|_{\sigma[\mathcal R]}\leq 1+\frac{d}{|q_0|}\max_{1\leq k\leq n}|q_k|$, where we have taken into account that  $|q|_{\sigma[\mathcal R]}=|q|$ for all $q\in\mathbb Q$. The latter implies $|z|_{\sigma[\mathcal R]}\leq n$ for any $n\in\mathbb N$ such that $1+\frac{d}{|q_0|}\max_{1\leq k\leq n}|q_k|\leq n$. Thus $z\in\mathcal F$. Similarly, $1/z\in\mathcal F$ (since $1/z\in\mathbb A(i)\setminus\{0\}$). Hence, $z\in\mathcal U$ as required. 
\item is equivalent to (i). 
\item Let $z+dx\in\mathbb A(i)$ and suppose (seeking a contradiction) that $dx\not= 0$. Thus $dx=z+dx-z\in\mathbb A(i)\setminus\{0\}$ implying $dx\in\mathcal U$ by (i), contradicting the assumption that $dx\in\mathcal I$. Conversely, $dx=0$ implies $z+dx\in\mathbb A(i)$ (trivially), because $z\in\mathbb A(i)$ by assumption. 
\item Let  $\sigma: \mathcal R(i)\mapsto\mathbb C$ be an isomorphism, $x$ be a non-zero infinitesimal in $\mathcal R$ and $\sigma(x)$ be the corresponding non-zero infinitesimal in $\mathbb C$ (as in the proof of Theorem~\ref{T: Main Result}). Let $z\in\mathbb C$ be a transcendental number. We extend $\{\sigma(x)\}$ and $\{z\}$ to two transcendence bases $\mathcal B_x:=\{\sigma(x)\}\cup\mathbb B_x$ and   $\mathcal B_z:=\{z\}\cup\mathbb B_z$, respectively, of the same field extension $\mathbb C: \mathbb Q$ (Section~\ref{S: Preliminaries: Ordered Fields}).  Let $\varphi: \mathbb Q(\mathcal B_x)\mapsto\mathbb Q(\mathcal B_z)$ denote the isomorphism defined by $\varphi(\sigma(x))=z$ and $\varphi[\mathcal B_x]=\mathcal B_z$. Since $\mathbb C\!: \!\mathbb Q$ is a Galois extension, $\varphi$ admits an  extension to an automorphism $\widehat{\varphi}: \mathbb C\mapsto\mathbb C$ (Hunderford~\cite{tHungerford}, Ch.VI) and/or  (Milne~\cite{jsMilne}, Ch.7); and we still have $\widehat{\varphi}(\sigma(x))=z$. Now, the isomorphism $\sigma_1: \mathcal R(i)\mapsto\mathbb C$,  defined by the formula $\sigma_1=\widehat{\varphi}\circ\sigma$, is what we are looking for, because $z=\sigma_1(x)$ is a non-zero infinitesimal. 
\end{Pf-enum}
\end{proof}
\quad\; The ``luxury'' for having infinitesimals in $\mathbb C$ (as any luxury in life or science) is not ``for free''. This is the content of the following remark.
\begin{remark}[Price to Pay]\label{R: Price to Pay} Let  $\sigma: \mathcal R(i)\mapsto\mathbb C$ be an isomorphism, where $ \mathcal R$ is a non-Archimedean real closed field of cardinality $\frak c$ (Theorem~\ref{T: Main Result}) and assume, in addition, that $\mathcal R$ contains a copy of $\mathbb R$. (The fields: $\mathbb R\{t\}$, $\mathbb R\langle t\rangle$, $\mathbb R[t]$, $^*\mathbb R$, $^*\!\mathbb A$ and $^*\mathbb R_\alpha$ in Example~\ref{Exs: Non-Archimedean Real Closed Fields of Cardinality c} satisfy these assumptions.) Notice that the field extension $\mathbb C: \sigma[\mathbb R]$ is transcendental (since $\mathcal R:\mathbb R$ is transcendental), in contrast to $\mathbb C: \mathbb R$, which is algebraic. Consequently, \emph{there is no automorphism} $\Phi: \mathbb C\mapsto\mathbb C$ such that $\Phi[\mathbb R]=\sigma[\mathbb R]$, because both field extensions are Galois extensions (Hunderford~\cite{tHungerford}, Ch.VI) and/or  (Milne~\cite{jsMilne}, Ch.7).
\end{remark}
\section{Examples of Non-Archimedean Real Closed Fields of Cardinality $\frak c=\card\mathbb R$}\label{S: Examples of Non-Archimedean Real Closed Fields of Cardinality c}

	\quad\; The complete the proof of Theorem~\ref{T: Main Result},  we have to show that there exist non-Archimedean real closed fields $\mathcal R$ of cardinality $\frak c$. The reader without background on non-standard analysis might decide to skip the last two or three examples with origin in non-standard analysis. As we mention in our Introduction, a single example, say $\mathcal R=\mathbb R\{t\}$, is enough to support Theorem~\ref{T: Main Result}.
\begin{examples}[Non-Archimedean Real Closed Fields of Cardinality $\frak c$]\label{Exs: Non-Archimedean Real Closed Fields of Cardinality c} All fields $\mathcal R$  listed below are real closed fields of cardinality $\frak c=\card\,\mathbb R$. The fields in the first two examples are Archimedean and the rest are non-Archimedean. Moreover, some of the fields contains a copy of $\mathbb R$, others not: 
\begin{enumerate}
\item We start with the field of real numbers $\mathbb R$.
\item Let $\alpha\in\mathbb R$ be a transcendental real number (say, $\alpha=\pi$ or $\alpha=e$) and $\mathcal B_\alpha$ be a subset of $\mathbb R$ such that $\{\alpha\}\cup\mathcal B_\alpha$ forms a transcendence basis of the field extension $\mathbb R\!:\!\mathbb Q$ (see the end of Section~\ref{S: Preliminaries: Ordered Fields}). Let $\mathbb R_\alpha$ be the real closure of $\mathbb Q(\mathcal B_\alpha)$ within $\mathbb R$. Then $\mathbb R_\alpha$ is an \emph{Archimedean real closed field of cardinality} $\frak c$, which does not contain a copy of $\mathbb R$, because $\alpha\notin\mathbb R_\alpha$. Thus $\mathbb R_\alpha(i)$ is a proper subfield of $\mathbb C$ (since $\mathbb R_\alpha$ is a proper subfield of $\mathbb R$) and, at the same time, $\mathbb R_\alpha(i)$ is isomorphic to $\mathbb C$ by Theorem~\ref{T: Main Result}. 
\item Let $\mathbb R_\alpha(t)$ stand for the field of rational functions with coefficients in $\mathbb R_\alpha$ (ordered such that $t$ is a positive infinitesimal) and $\widehat{\mathbb R}_\alpha$ is a real closure of $\mathbb R_\alpha(t)$. Then $\widehat{\mathbb R}_\alpha$ is a non-Archimedean real closed field of cardinality $\frak c$. Hence, $\widehat{\mathbb R}_\alpha(i)$ is isomorphic to $\mathbb C$ by Theorem~\ref{T: Main Result}
\item The field $\mathbb R\{t\}$ of \emph{Puiseux series} (also known as \emph{Puiseux-Newton series}) with real coefficients (Prestel~\cite{aPrestel})
 is defined by
\[
\mathbb R\{t\}=:\Big\{\sum_{n=\mu}^{\infty}a_n(\sqrt[m]{t})^{n} : \mu\in\mathbb Z,\; m\in\mathbb N,\; a_n\in\mathbb R\Big\}.
\] 
Puiseux series were introduced by Isaac Newton~\cite{iNewton1736} in 1736 and used for calculation of planet's orbits. Later $\mathbb R\{t\}$ was rediscovered by Victor Puiseux~\cite{vPuiseux1850}-\cite{vPuiseux1851}.  Clearly, $\mathbb R\{t\}$ has cardinality $\frak c$. $\mathbb R\{t\}$ is, in fact, a \emph{real closure of the field $\mathbb R\bra t^\mathbb Z\ket$ of Laurent series} with real coefficients hence, it is a real closed field.
\item The field of \emph{Levi-Civit\'{a}~\cite{tLeviCivita} series} with real coefficients is defined by
\[
\mathbb R\langle t^{\mathbb R}\rangle=:\Big\{\sum_{r\in\mathbb R} a_r t^{r} : a_r\in\mathbb R \text{ and } \supp(\sum_{r\in\mathbb R} a_r t^{r})  \text{ is a left-finite set}\Big\}.
\]
where $ \supp(\sum_{r\in\mathbb R} a_r t^{r})=\{r\in\mathbb R: a_r\not=0\}$. Recall that a subset $S$  of $\mathbb R$ is a \emph{left-finite} if $\{s\in S: s\leq r\}$ is finite for all $r\in\mathbb R$. The field $\mathbb R\langle t^{\mathbb R}\rangle$ is a real closed field of cardinality $\frak c$. Levi-Civit\'{a} series are used nowadays for symbolic calculations of derivatives of real functions in computers (Shamseddine, Berz~\cite{kShamseddine2010}).
\item The field of \emph{Hahn's (generalized) series with coefficients in $\mathbb R$ and valuation group} $(\mathbb R, +, <)$ is defined by
\[
\mathbb R[t^{\mathbb R}]=: \Big\{\sum_{r\in\mathbb R} a_rt^r : a_r\in\mathbb R \text{ and } \supp(\sum_{r\in\mathbb R} a_r t^{r})  \text{ is a well ordered set}\Big\},
\]
(Hahn~\cite{hHahn}). For example, $1+t^{1/2}+t^{2/3}+t^{3/4}+\dots\in \mathbb R[t^{\mathbb R}]\setminus\mathbb R\langle t^{\mathbb R}\rangle$, since $\lim_{n\mapsto\infty}\frac{n}{n+1}=1\not=\infty$ hence, the set $\{0, 1/2, 2/3,\dots\}$ is well-ordered, but not left-finite. $\mathbb R[ t^{\mathbb R}]$ is a real closed field of cardinality $\frak c$. Hahn' series were developed in connection with Hilbert's Second Problem.
\item Robinson's field of non-standard numbers $^*\mathbb R=\mathbb R^\mathbb N/\!\!\sim_\mathcal U$. Here $\mathbb R^\mathbb N$ stands for the ring of the sequences of real numbers, $\mathcal U$ is a free ultrafilter on $\mathbb N$ and 
$(a_n)\sim_\mathcal U (b_n)$ on $\mathbb R^\mathbb N$ if $\{n\in \mathbb N: a_n=b_n\}\in\mathcal U$. We denote by $\langle a_n\rangle$ the equivalence class of $(a_n)$.  We embed $\mathbb R\subset{^*\mathbb R}$ under  $r\mapsto \bra r, r,\dots\ket$. Consequently, $\card(^*\mathbb R)=\card(\mathbb R)=\frak c$. The field $^*\mathbb R$ is a real closed field by Transfer Principle (Davis~\cite{mDavis}, p.28), since $\mathbb R$ is a real closed field. Moreover, $^*\mathbb R$ is non-Archimedean field; in particular, $\bra n\ket$ is infinitely large positive number, because $\{n\in\mathbb N: m<n\}\in\mathcal U$ for any (fixed) $m\in\mathbb N$. We should mention that although A. Robinson~\cite{aRob66} is commonly accepted as the creator of non-standard analysis, the ultrapower construction of $^*\mathbb R$ mentioned above appears, first, in Luxemburg~\cite{wLux62} (see also Stroyan \& Luxemburg~\cite{StroLux76}). For a presentations on the topic we mention as well (Loeb \& Wolff~\cite{LoebWolff}) and (Lindstr\o m~\cite{tLindstrom}).  
\item Let $\mathbb A$ denote the field of real algebraic numbers. Recall that $\mathbb A$ is an Archimedean real closed field of cardinality $\aleph_0$. Let $^*\!\mathbb A=\mathbb A^\mathbb N/\!\!\sim_\mathcal U$ stand for the non-standard extension of $\mathbb A$. The field $^*\!\mathbb A$ is a \emph{non-Archimedean real closed field of cardinality} $\frak c$ (the help of CH is required). Moreover, $^*\!\mathbb A$ contains (non-canonically) a copy of $\mathbb R$. Indeed, every subfield $\mathcal A$ of $^*\!\mathbb A$, which is a maximal subring of $\mathcal F(^*\!\mathbb A)$ (a field of representatives) is isomorphic to $\mathbb R$, because $\mathcal F(^*\!\mathbb A)/\mathcal I(^*\!\mathbb A)$ is isomorphic of $\mathbb R$ (Davis~\cite{mDavis}, Thm.2.5, p.52). Here $\mathcal F(^*\!\mathbb A)$ and $\mathcal I(^*\!\mathbb A)$ stand for the sets of finite and infinitesimal elements of $^*\!\mathbb A$, respectively.  The existence of such a maximal field $\mathcal A$ is guaranteed by Zorn lemma.
\item Similarly to above example,  the non-standard extension $^*\mathbb R_\alpha$ of $\mathbb R_\alpha$ is a \emph{non-Archimedean real closed field of cardinality} $\frak c$, which (like $^*\!\mathbb A$, but unlike $\mathbb R_\alpha$) contains a copy of $\mathbb R$.
\item Let $\rho\in{^*\mathbb R}$ be a (fixed) positive infinitesimal. \emph{Robinson's field of asymptotic numbers} is defined by $^\rho\mathbb R=\mathcal M_\rho/\mathcal N_\rho$, where
\begin{align}
&\mathcal M_\rho=\{x\in{^*\mathbb R}: |x|\leq\rho^{-n} \text{ for some } n\in\mathbb N\},\notag\\
&\mathcal N_\rho=\{x\in{^*\mathbb R}: |x|<\rho^{n} \text{ for all } n\in\mathbb N\},\notag
\end{align}
(Robinson~\cite{aRob73}, Lightstone \& Robinson~\cite{LiRob}). Notice that $e^{1/\rho}\notin\mathcal M_\rho$ and $e^{-1/\rho}\in\mathcal N_\rho$. The field $^\rho\mathbb R$ is real closed and spherically complete (Luxemburg~\cite{wLux76},\, Pestov~\cite{vPestov}).
More recently, the field  $^\rho\mathbb C$ had beed used in the role of \emph{fields of scalars of the algebra of generalized  functions of Colombeau type} (Oberguggenberger \&Todorov~\cite{OberTod98}, Todorov \&Vernaeve~\cite{TodVern08},  Todorov~\cite{tdTodAxioms10}-\cite{tdTodorov2024}). For more details we refer to Section~\ref{S: Infinitesimals in Analysis}.

\item Let us replace $\mathbb R$ in the fields $\mathbb R\{t\}$, $\mathbb R\bra t^\mathbb R\ket$, $\mathbb R[t^\mathbb R]$ by any real closed field (Archimedean or not) of cardinality $\frak c$. The result will be another non-Archimedean real closed field of cardinality $\frak c$. For example, the fields:  $\mathbb R_\alpha\{t\}$, $\mathbb R_\alpha\bra t^\mathbb R\ket$, $\mathbb R_\alpha[t^\mathbb R]$, $^*\mathbb R\{t\}$, $^*\mathbb R\bra t^\mathbb R\ket$, $^*\mathbb R[t^\mathbb R]$, $^*\!\mathbb A\{t\}$, $^*\!\mathbb A\bra t^\mathbb R\ket$, $^*\!\mathbb A[t^\mathbb R]$, $^*\mathbb R_\alpha\{t\}$, $^*\mathbb R_\alpha\bra t^\mathbb R\ket$, $^*\mathbb R_\alpha[t^\mathbb R]$, $^\rho\mathbb R\{t\}$, $^\rho\mathbb R\bra t^\mathbb R\ket$ and $^\rho\mathbb R[t^\mathbb R]$, etc., are \emph{non-Archimedean real closed field of cardinality} $\frak c$.
\end{enumerate}
\end{examples}

	It is clear that all fields listed above are pairwise non-isomorphic. Let $\mathcal R$ be any of the real closed fields (Archimedean or not) of cardinality $\frak c$ listed in  Examples~\ref{Exs: Non-Archimedean Real Closed Fields of Cardinality c}. Then  $\mathbb C\cong\mathcal R(i)$ by Corollary~\ref{C: Two Fields}, where $\cong$ stands for \emph{isomorphic}, since $\mathbb C=\mathbb R(i)$. In particular,
\begin{equation}\label{E: C, *C, rhoC}
\mathbb C\cong{^*\mathbb C}\cong{^\rho\mathbb C},
\end{equation}
where ${^*\mathbb C}={^*\mathbb R}(i)$ stands for the \emph{non-standard extension of} $\mathbb C$ and ${^\rho\mathbb C}={^\rho\mathbb R}(i)$ is the field of \emph{Robinson's complex asymptotic numbers}. As far as we can detect, the isomorphism (\ref{E: C, *C, rhoC}) had not beed noticed in the literature so far.
\section{Uniqueness up to Isomorphism}\label{S: Uniqueness up to Isomorphism}

	\quad\; Some of the numerous fields listed in Examples~\ref{Exs: Non-Archimedean Real Closed Fields of Cardinality c} are actually, identical or isomorphic under some additional assumptions. For example, if $\rho_1/\rho_2$ is a finite, but non-infinitesimal element of $^*\mathbb R$ for some positive infinitesimals $\rho_1, \rho_2\in{^*\mathbb R}$, we have $\mathcal M_{\rho_1}=\mathcal M_{\rho_2}$. The latter implies $\mathcal N_{\rho_1}=\mathcal N_{\rho_2}$, since $\mathcal M_{\rho_1}$ and $\mathcal M_{\rho_2}$ are convex rings. In this case $^{\rho_1}\mathbb R$ and $^{\rho_2}\mathbb R$ are not just isomorphic, but identical. In this section we impose CH in addition to ZFC (Section~\ref{S: Set-Theoretical Framework}) and show that many of the fields, listed in Examples~\ref{Exs: Non-Archimedean Real Closed Fields of Cardinality c}, are actually, isomorphic. We hope that the discussion in this section leads to a simplification and some conceptual clarity.
\begin{enumerate}
\item Recall that the fields $^*\mathbb R$ are an algebraically saturated. For the original result we refer to Luxemburg~\cite{wLux69}, p.18-86) and for a presentation we mention also (Lindstr\o m~\cite{tLindstrom}, Thm.1.2.5, p.12). Consequently, $\mathbb R^\mathbb N/\!\!\sim_{\mathcal U_1}$ and  $\mathbb R^\mathbb N/\!\!\sim_{\mathcal U_2}$ are isomorphic by Theorem~\ref{T: Erdos, Gillman and Henriksen} (which requires CH) for every two free ultrafilters $\mathcal U_1$ and $\mathcal U_2$ on $\mathbb N$. 
\item In contrast to the above, the fields $^\rho\mathbb R$ are not saturated and Theorem~\ref{T: Erdos, Gillman and Henriksen}  is not directly applicable. However, each of the fields $^\rho\mathbb R$ turns out to be isomorphic to ${^*\mathbb R}[t^\mathbb R]$, the Hahn field with coefficients in $^*\mathbb R$ and valuation group $(\mathbb R, +, <)$. Thus ${^{\rho_1}\mathbb R}$ and ${^{\rho_2}\mathbb R}$ are isomorphic for every two positive infinitesimals, $\rho_1$ and $\rho_2$ in $^*\mathbb R$ (Todorov\,\&Wolf~\cite{TodWolf04}). The latter requires CH. 
\item We have two \emph{self-embeddings} $^*\mathbb R\emb{^\rho\mathbb R}$ and $^\rho\mathbb R\emb{^*\mathbb R}$ (both non-canonical). 
\begin{itemize}
\item The embedding $^*\mathbb R\emb{^\rho\mathbb R}$ follows from the isomorphism between ${^*\mathbb R}[t^\mathbb R]$ and $^\rho\mathbb R$ mentioned above.
\item To show the embedding $^\rho\mathbb R\emb{^*\mathbb R}$, we exploit  the fact that $^\rho\mathbb R$ is isomorphic to any  subfield $\mathbb M$ of $^*\mathbb R$, which is maximal subring of $\mathcal M _\rho$ (a field of representatives). The existence of such a maximal field $\mathbb M$ is guaranteed by Zorn lemma (Todorov\,\&Wolf~\cite{TodWolf04}, Todorov~\cite{tdTod06Notes}, \S14). 
\end{itemize}
\item On the ground of isomorphism ${^\rho\mathbb R}\cong{^*\mathbb R}[t^\mathbb R]$, we can eliminate $^\rho\mathbb R$ and write 
\begin{align}
&\mathbb R\{t\}\subset\mathbb R\bra t^\mathbb R\ket\subset\mathbb R[t^\mathbb R]\subset{^*\mathbb R}[t^\mathbb R]\emb{^*\mathbb R},\notag\\
&\mathbb R\{t\}\subset\mathbb R\bra t^\mathbb R\ket\subset\mathbb R[t^\mathbb R]\emb{^*\mathbb R}\subset{^*\mathbb R}[t^\mathbb R],\notag
\end{align}
where $\subset$ stands for ``subfield'' (hence, ``subset'') and $\emb$ stands for ``field embedding''. 
\end{enumerate}
\section{$\kappa$-Complex Numbers}\label{S: kappa-Complex Numbers}

	\quad\; We generalize the result of Theorem~\ref{T: Main Result} and show that every field $\mathbb C_\kappa$ of $\kappa$-complex numbers (defined below) contains non-zero infinitesimals. 

\begin{definition}[$\kappa$-Complex Numbers]\label{D: kappa-Complex Numbers} Let $\kappa$ be an infinite cardinal and let $\mathbb C_{\kappa}$ be an algebraically closed field of zero characteristic and cardinality $\kappa_+$. We refer to the elements of $\mathbb C_{\kappa}$ as $\kappa$-complex numbers.
\end{definition}
	In the next theorem we show the existence of $\mathbb C_{\kappa}$ for every infinite cardinal $\kappa$. For the concept of \emph{$\kappa_+$-good free ultrafilter} (used below) we refer to Chang \& Keisler~\cite{ChangKeisler} and for an exposition on the topic we also mention  (Lindstr\o m~\cite{tLindstrom}, Appendix). 
\begin{theorem}[Generalized Fundamental Theorem of Algebra]\label{T: Generalized Fundamental Theorem of Algebra}
 For every infinite cardinal $\kappa$  there exists a unique, up to isomorphism, algebraically closed field of zero-characteristic and  cardinality $\kappa_+$. Consequently $\mathbb C_\kappa$ exists for every infinite cardinal $\kappa$.
 \end{theorem}
 \begin{proof} $\mathbb R$ is a real closed field, because every polynomial with real coefficients is a continuous function on $\mathbb R$ hence, if the degree of such a polynomial is odd, its graph must intersect the $x$-axis at least once. Thus $\mathbb R(i)$ is an algebraically closed field of zero characteristic by Artin-Schreier~\cite{ArtinSchreier27}  theorem - the so called fundamental theorem of algebra. Next, let $I$ is a set of cardinality $\kappa$ and $\mathcal V$ be a $\kappa_+$-good free ultrafilter on $I$.  Let  $\mathbb R^I/\!\sim_\mathcal V$ stand for the corresponding non-standard extension of $\mathbb R$. Thus (whether $\kappa=\aleph_0$ or $\kappa>\aleph_0$)  the field $\mathbb R^I/\!\sim_\mathcal V$ is real closed by transfer principle (Davis~\cite{mDavis}, p.28), since $\mathbb R$ is real closed and $\mathbb R^I/\!\sim_\mathcal V$ is a non-standard extension of $\mathbb R$. Also, $\mathbb R^I/\!\sim_\mathcal V$ is of cardinality $\kappa_+$, because the ultrafilter  $\mathcal V$ is $\kappa$-good by assumption.   It is customary to use the simpler notation $^*\mathbb R$ (along with  $\card(^*\mathbb R)=\kappa_+$) instead of  $\mathbb R^I/\!\sim_\mathcal V$. Using this notation,  $^*\mathbb R(i)$ must be algebraically closed field of zero characteristic by Artin-Schreier theorem and $^*\mathbb R(i)$ is of cardinality $\kappa_+$, since $\card(^*\mathbb R)=\kappa_+$. These properties of  $^*\mathbb R(i)$ determines $^*\mathbb R(i)$ uniquely, up to isomorphism, by Corollary~\ref{C: Uncountable Cardinality} applied for characteristic zero. In particular, $^*\mathbb R(i)$ is isomorphic to $\mathbb C_\kappa$ proving the existence of $\mathbb C_\kappa$.
 \end{proof}
 \begin{remark}[Countable Fields]\label{R: Countable Fields} Let $\mathbb A(i)$ denote the field of the algebraic numbers in $\mathbb C$ (here $\mathbb A$ stands for the field of the real algebraic numbers). The field $\mathbb A(i)$ is algebraically closed, of zero characteristic and $\card(\mathbb A(i))=\aleph_0$. Notice however, that the countable fields are excluded in the above theorem; the reason is that they do not satisfy the condition for ``uniqueness up to isomorphism''; two countable algebraically closed fields of zero characteristic might or might not be isomorphic.
 \end{remark}
 \begin{theorem} Each of the fields $\mathbb C_\kappa$ contains non-zero infinitesimals in the sense that $\mathbb C_\kappa$ contains a non-Archimedean subfields. 
  \end{theorem}
 \begin{proof} We already know that $\mathbb C_\kappa$ contains non-zero infinitesimals in the case $\kappa=\aleph_0$, since $\mathbb C_{\aleph_0}$ is isomorphic to $\mathbb C$ (Section~\ref{S: Infinitesimals in Algebra}). In the case $\kappa>\aleph_0$ we observe that $\mathbb C_\kappa$ is isomorphic to $\mathcal R(i)$ for every real closed field $\mathcal R$ of cardinality $\kappa_+$. Hence, $\mathcal R$ must be non-Archimedean as totally ordered field  of cardinality greater than $\frak c$. Thus $\mathcal R$ can be embedded as a subfield of $\mathbb C_\kappa$ supplying  $\mathbb C_\kappa$ with non-zero infinitesimals.
\end{proof}
 \begin{remark}[Warning about Notation]\label{R: Warning about Notation} The asterisk notation (like any notation in mathematics) is less than perfect. For example, the fields $\mathbb R^\mathbb N/\!\!\sim_\mathcal U$ and  $\mathbb R^I\!\!/\sim_\mathcal V$ are certainly non-isomorphic in the case $\card(I)>\aleph_0$. Indeed, we have   $\card(\mathbb R^\mathbb N/\!\!\sim_\mathcal U)=\frak c$, while $\card(\mathbb R^I/\!\!\sim_\mathcal V)=\kappa_+>\frak c$, where $\kappa=\card(I)$. Still we usually use the same notation, $^*\mathbb R$ for both. The same hold for $^*\mathbb C$ and the corresponding fields of asymptotic numbers $^\rho\mathbb R$ and $^\rho\mathbb C$. For example, we have $\card(^*\mathbb C)=\card(^\rho\mathbb C)=\frak c$ in (Oberguggenberger \& Todorov~\cite{OberTod98}) and $\card(^*\mathbb C)=\card(^\rho\mathbb C)=\frak c_+$ in (Todorov \& Vernaeve~\cite{TodVern08}). To avoid misunderstanding, the readers is adviced to pay attention to the choice of the saturation (hence, cardinality) of $^*\mathbb R$ usually specified at the very beginning of the text under reading.
  \end{remark}
 \section{Infinitesimals in Analysis}\label{S: Infinitesimals in Analysis}
 
 	\quad\;So far we discuss the infinitesimals in $\mathbb C$ (Definition~\ref{D: Abstract Definition of C}) and  infinitesimals in $\mathbb C_\kappa$ (Section~\ref{S: kappa-Complex Numbers}) mostly from algebraic point of view. In our last section we supply $\mathbb C_\kappa$ hence, also $\mathbb C$, with different, non-homeomorphic in general, topologies and shortly discuss standard, non-standard and non-standard asymptotic analysis in comparison.
 \begin{definition}[Topology on $\mathbb C_\kappa$]\label{D: Topology on C_kappa}  Let $\kappa$ be an infinite cardinal and $\mathbb C_\kappa$ stand for the field of $\kappa$-compex numbers (Definition~\ref{D: kappa-Complex Numbers}). Let $\mathcal R$ be  a real closed field of cardinality $\kappa_+$ and $\sigma: \mathcal R(i)\mapsto\mathbb C_\kappa$ stand for a field isomorphism (Section~\ref{S: kappa-Complex Numbers}).
  \begin{D-enum}
 
 \item We define the absolute value $|\cdot|_{\sigma[\mathcal R]}: \mathbb C_\kappa\mapsto\mathbb C_\kappa$, by $|\sigma(x+iy|_{\sigma[\mathcal R]}=\sigma(\sqrt{x^2+y^2})$ for all $x, y\in\mathcal R$ and supply $\mathbb C_\kappa$ with the topology generated by $|\cdot|_{\sigma[\mathcal R]}$. We denote the corresponding topological field by $(\mathbb C_\kappa,\, {|\!\cdot\!|}_{\sigma[\mathcal R]})$.
 \item In the particular case $\kappa=\aleph_0$ we shall write simply, $(\mathbb C,\, {|\!\cdot\!|}_{\sigma[\mathcal R]})$, instead of $(\mathbb C_{\aleph_0},\, {|\!\cdot\!|}_{\sigma[\mathcal R]})$
 (since $\mathbb C_{\aleph_0}$ is isomorphic to $\mathbb C$ and the latter determines $\mathbb C$ uniquely up to isomorphism - Definition~\ref{D: Abstract Definition of C}). 
 \item We say that $z\in\mathbb C_\kappa$ (in particular, $z\in\mathbb C$) is an \emph{infinitesimal relative to the absolute value} $|\cdot|_{\sigma[\mathcal R]}$\;  if $0\leq |z|_{\sigma[\mathcal R]}< 1/n$ for all $n\in\mathbb N$. Similarly, $z$ \emph{is finite relative to} $|\cdot|_{\sigma[\mathcal R]}$\; if $|z|_{\sigma[\mathcal R]}\leq n$ for some $n\in\mathbb N$ and $z$ \emph{is infinitely large relative to} $|\cdot|_{\sigma[\mathcal R]}$\; if $n<|z|_{\sigma[\mathcal R]}$ for all $n\in\mathbb N$.
\end{D-enum}
 \end{definition}
 	The absolute value $|\cdot|_\mathcal R$ determines on $\mathbb C_\kappa$  the product topology  inherited from the order topology on $\sigma[\mathcal R]$. In addition, $|\cdot|_{\sigma[\mathcal R]}$ can be used for localizing the infinitesimals within $\mathbb C$. 
 \begin{lemma}[Infinitesimals Relative to the Absolute Value] $\mathbb C_\kappa$ is free of non-zero infinitesimals relative to $|\cdot|_{\sigma[\mathcal R]}$  if and only if $\mathcal R$ is an Archimedean field if and only if $\sigma[\mathcal R]$ is an Archimedean field  if and only if $\kappa=\aleph_0$. 
  \end{lemma}
 \begin{proof}For the range of the absolute value we have $\ran(|\cdot|_{\sigma[\mathcal R]})=\{x\in\sigma[\mathcal R]: x\geq 0\}$. Thus $\ran(|\cdot|_{\sigma[\mathcal R]})$ is free of positive infinitesimals \ifff  $\sigma[\mathcal R]$ is free of non-zero infinitesimals \ifff  $\sigma[\mathcal R]$
 is an Archimedean field \ifff $\mathcal R$ is an Archimedean field. Finally, the latter is equivalent to $\kappa=\aleph_0$, since $\card(\mathbb C_\kappa)=\aleph_1=\frak c=\card(\mathcal R)$ (Section~\ref{S: Set-Theoretical Framework}). 

 \end{proof}
 \begin{example} We go through the examples in Examples~\ref{Exs: Non-Archimedean Real Closed Fields of Cardinality c} and their counterparts in Section~\ref{S: kappa-Complex Numbers}.  
 \begin{enumerate}
  \item $\mathbb C$  is \emph{free of non-zero infinitesimals} (hence, free of infinitely large elements) \emph{relative to} $|\cdot|_{\sigma[\mathcal R]}$ in the particular cases $\mathcal R=\mathbb R$ and $\mathcal R=\mathbb R_\alpha$, where $\alpha$ is a transcendental real number (Examples~\ref{Exs: Non-Archimedean Real Closed Fields of Cardinality c}).
 \item  $(\mathbb C,\; {|\!\cdot\!|}_{\sigma[\mathbb R]})$, commonly denoted, simply by $\mathbb C$ or $\mathbb R(i)$, is the only topological field among those in the form $(\mathbb C,\, {|\!\cdot\!|}_{\sigma[\mathcal R]})$, which is a Banach algebra. The latter follows from the fact that $\mathbb R$ hence, $\sigma[\mathbb R]$, is Dedekind complete ($\sup$-complete).  
 \item Let $\mathcal R$ be any of the  real closed fields in Examples~\ref{Exs: Non-Archimedean Real Closed Fields of Cardinality c} except the fields $\mathcal R=\mathbb R$ and $\mathcal R=\mathbb R_\alpha$ (mentioned above). Then $\mathbb C$  \emph{does contain non-zero infinitesimals} (hence, infinitely large elements) \emph{relative to} $|\!\cdot\!|_{\sigma[\mathcal R]}$ (the fields $\mathcal R=\mathbb R$ and $\mathcal R=\mathbb R_\alpha$ are the only Archimedean fields in Examples~\ref{Exs: Non-Archimedean Real Closed Fields of Cardinality c}).
 \item Let $^*\mathbb R$ be the non-standard extension of $\mathbb R$ of cardinality $\frak c$ (Examples~\ref{Exs: Non-Archimedean Real Closed Fields of Cardinality c}).  The topological field $(\mathbb C,\, {|\!\cdot\!|}_{\sigma[^*\mathbb R]})$ is commonly denoted by $^*\mathbb C$ with the restriction $\card(^*\mathbb C)=\frak c$ or simply, by ${^*\mathbb R}(i)$, where $\card(^*\mathbb R)=\frak c$. It is the only topological field among those in Examples~\ref{Exs: Non-Archimedean Real Closed Fields of Cardinality c}, which is $\frak c$-saturated in the sense that every nested sequence \emph{of open disks} (more generally, \emph{of internal sets}) of $\mathbb C$ relative to $ {|\!\cdot\!|}_{\sigma[^*\mathbb R]}$ has non-empty intersection (Lindstr\o m~\cite{tLindstrom}, Loeb~\cite{LoebWolff}).  More generally, let $\mathbb C_\kappa$ be a field of $\kappa$-complex numbers of cardinality $\kappa_+$  (Definition~\ref{D: kappa-Complex Numbers}). Then the topological field $(\mathbb C_\kappa,\, {|\!\cdot\!|}_{\sigma[^*\mathbb R]})$, is commonly denoted also by $^*\mathbb C$ with the restriction $\card(^*\mathbb C)=\kappa_+$ or simply, by ${^*\mathbb R}(i)$, where $\card(^*\mathbb R)=\kappa_+$ (Remark~\ref{R: Warning about Notation}). The field $^*\mathbb C=(\mathbb C_\kappa,\, {|\!\cdot\!|}_{\sigma[^*\mathbb R]})$ is $\kappa_+$-saturated in the sense that  every family of cardinality $\kappa$ \emph{of open disks} (and more generally, \emph{of internal sets}) of $\mathbb C_\kappa$ relative to $ {|\!\cdot\!|}_{\sigma[^*\mathbb R]}$ with the finite intersection property has non-empty intersection (Lindstr\o m~\cite{tLindstrom}, Loeb~\cite{LoebWolff}). 
 \item Let $^\rho\mathbb R$ stand for Robinson's field of asymptotic numbers of cardinality $\frak c$ (Examples~\ref{Exs: Non-Archimedean Real Closed Fields of Cardinality c}).  The topological field $(\mathbb C,\, {|\!\cdot\!|}_{\sigma[^\rho\mathbb R]})$ is commonly denoted by $^\rho\mathbb C$ with the restriction $\card(^\rho\mathbb C)=\frak c$ or equivalently, by ${^\rho\mathbb R}(i)$, where $\card(^\rho\mathbb R)=\frak c$. The field $^\rho\mathbb C$ is non-saturated, but it is Cantor complete in the sense that every  nested sequence of closed  disks of $\mathbb C$ relative to ${|\!\cdot\!|}_{\sigma[^\rho\mathbb R]}$ has non-empty intersection (Todorov \&Wolf~\cite{TodWolf04}). Similarly, the topological field $(\mathbb C_\kappa,\, {|\!\cdot\!|}_{\sigma[^\rho\mathbb R]})$ is commonly denoted  also by $^\rho\mathbb C$ with the restriction $\card(^\rho\mathbb C)=\kappa_+$ or equivalently, by ${^\rho\mathbb R}(i)$, where $\card(^\rho\mathbb R)=\kappa_+$. As before, $^\rho\mathbb C$ is Cantor complete, but non-saturated.
 \item As we already know, from purely algebraic point of view we have the isomorphism ${\mathbb C_\kappa}\cong{^*\mathbb C}\cong{^\rho\mathbb C}$ in the case $\card(^*\mathbb C)=\card(^\rho\mathbb C)=\kappa_+$. In particular, ${\mathbb C}\cong{^*\mathbb C}\cong{^\rho\mathbb C}$ in the case $\card(^*\mathbb C)=\card(^\rho\mathbb C)=\frak c$. In contrast, the topological fields  $(\mathbb C,\; {|\!\cdot\!|}_{\sigma[\mathbb R]})$, $(\mathbb C,\, {|\!\cdot\!|}_{\sigma[^*\mathbb R]})$ and $(\mathbb C,\, {|\!\cdot\!|}_{\sigma[^\rho\mathbb R]})$ are \emph{pairwise non-homeomorphic}. Indeed, $(\mathbb C,\; {|\!\cdot\!|}_{\sigma[\mathbb R]})$ is Dedekind complete (sup-complete) unlike the other two, which are not (because they both contain non-zero infinitesimals relative to their absolute values). On the other hand, the last two fields are also non-homeomorphic, because $(\mathbb C,\, {|\!\cdot\!|}_{\sigma[^*\mathbb R]})$ is saturated unlike $(\mathbb C,\, {|\!\cdot\!|}_{\sigma[^\rho\mathbb R]})$, which is non-saturated. Similarly, the topological fields $(\mathbb C_\kappa,\, {|\!\cdot\!|}_{\sigma[^*\mathbb R]})$  and $(\mathbb C_\kappa,\, {|\!\cdot\!|}_{\sigma[^\rho\mathbb R]})$ are non-homeomorphic, because the first is saturated and the second is not.  These topological fields behave quite differently in complex and functional analysis, which is the topic of the next remark. 
 \end{enumerate}
 \end{example}
 \begin{remark}[Standard vs. Non-Standard  Analysis]\label{R: Standard vs. Non-Standard  Analysis}
 \begin{enumerate}
 \item As we know, the topological field $(\mathbb C,\, {|\!\cdot\!|}_{\sigma[\mathbb R]})$ is the common framework of (standard) \emph{complex and functional analysis}. Its success is due to the fact that $\mathbb R$ is Dedekind complete hence, $\mathbb C$  is a Banach algebra.
 \item The field ${^*\mathbb C}=(\mathbb C,\, {|\!\cdot\!|}_{\sigma[^*\mathbb R]})$ and its ``larger'' counterpart ${^*\mathbb C}=(\mathbb C_\kappa,\, {|\!\cdot\!|}_{\sigma[^*\mathbb R]})$ are the framework of \emph{non-standard complex and non-standard functional analysis} (Davis~\cite{mDavis},    Lindstr\o m~\cite{tLindstrom}, Loeb~\cite{LoebWolff}).  Due to the Transfer Principle (Davis~\cite{mDavis}, Ch.1, Thm.7.3), every result of standard analysis (expressed in terms of the field $\mathbb C=(\mathbb C,\, {|\!\cdot\!|}_{\sigma[\mathbb R]})$) can be also expressed equivalently in the terms of the fields ${^*\mathbb C}=(\mathbb C,\, {|\!\cdot\!|}_{\sigma[^*\mathbb R]})$ or ${^*\mathbb C}=(\mathbb C_\kappa,\, {|\!\cdot\!|}_{\sigma[^*\mathbb R]})$, i.e. in the language of non-standard analysis. In addition, the non-standard version of  the statements \emph{are simpler} than their standard counterpart in the sense that they involve less quantifiers (Cavalcante~\cite{rCavalcante}). Here is a typical example: Let $c$ be an accumulating point of $D\subseteq\mathbb C$ and $f: D\mapsto\mathbb C$ be a complex-valued function. Let $L\in\mathbb C$. Then the following are equivalent:
 \begin{align}
  & (\forall dx\in\mathcal I(^*\mathbb R))(c+dx\in{^*\!D}\Rightarrow {^*\!f}(c+dx)-L\in\mathcal I(^*\mathbb R)),\label{F: Limit}\\
 & (\forall\varepsilon\in\mathbb R_+)(\exists\gamma\in\mathbb R_+)(\forall z\in D)(|z-c|<\gamma \Rightarrow | f(z)-L |<\varepsilon),\\
&\lim_{x\mapsto c}f(z)=L.
 \end{align}
 Here $\mathcal I(^*\mathbb R)$ stands for the set of infinitesimals in $^*\mathbb R$ and $^*\!D$ and $^*\!f$ are the non-standard extension of $D$ and $f$, respectively. 
 \item Readers with interest to history of calculus might notice that (\ref{F: Limit}) very much resembles a criterion for convergence in the old Leibniz/Euler infinitesimal calculus. This especially become obvious if we drop the asterisks in front of $^*\mathbb R, {^*\!D}$ and $f$ and verbalize the statement in language of 17-19 centuries mathematics. We should be not surprise that the calculus was first invented in terms of infinitesimals requiring a single quantifier (rather than three non-commuting ones). 

 \item A more recent example of taking advantage of the simplicity of non-standard methods offers the \emph{Bernstein--Robinson~\cite{BernsteinRobinson66} solution of the open invariant subspace problem for polynomially compact operators on Hilbert space} (for a presentation see also  Davis~\cite{mDavis}, Ch.5, \S3). A standard solution was produced a decay later (Lomonosov~\cite{vLomonosov}). 
 \item Another example of efficiency of non-standard methods comes from the recent \emph{solution of the Hilbert fifth problem for local groups} (Goldbring~\cite{iGoldbring}). As far as we can tell, a standard solution of this problem has not been produced so far.  As in the previous example, the reducing of the numbers of quantifiers (compared with those in the earlier attempts to solve the same problem by standard methods) plays a decisive role.  
 \item As we noted already every statement in standard analysis has an equivalent counterpart in non-standard analysis. The converse however, is not always true: there are results in non-standard analysis without counterpart in standard analysis. For example,   Dirac delta-functional  appears as a pointwise function $\delta: {^*\mathbb R}\mapsto{^*\mathbb C}$ within an algebra of generalized function of Colombeau type (Vernaeve~\cite{HVern03}). The latter is impossible in standard analysis, since $\delta$  takes infinitely large values for some infinitesimal numbers in $^*\mathbb R$. Similar result appears in non-standard asymptotic analysis discussed below. Let us recall that an associative commutative differential algebra $G$ is called \emph{ algebras of generalized  functions of Colombeau type} if $G$ contains a copy of Schwartz distributions $\mathcal D^\prime$ (as a differential vector subspace) and such that the product in $G$ reduces to the usual point-wise product on the ring of $C^\infty$-functions in $\mathcal D^\prime$ (Colombeau~\cite{jCol84a}). Every such algebra \emph{solves the problem of multiplication of Schwartz distributions}, since $\mathcal D^\prime$ itself does not admit associative product (Schwartz~\cite{lSchwartz54}). 
 \item The field ${^\rho\mathbb C}=(\mathbb C,\, {|\!\cdot\!|}_{\sigma[^\rho\mathbb R]})$ and its ``larger'' counterpart ${^\rho\mathbb C}=(\mathbb C_\kappa,\, {|\!\cdot\!|}_{\sigma[^\rho\mathbb R]})$ are the  framework of \emph{non-standard asymptotic analysis} (Robinson~\cite{aRob73}, Lightstone \& Robinson~\cite{LiRob}). As mentioned already, the field  $^\rho\mathbb C$ had beed used in the role of \emph{fields of scalars of algebras of generalized  functions of Colombeau type} (Oberguggenberger \&Todorov~\cite{OberTod98}, Todorov \&Vernaeve~\cite{TodVern08},  Todorov~\cite{tdTodAxioms10}-\cite{tdTodorov2024}). 
 \end{enumerate}
 \end{remark}
\begin{remark}[Summary]\label{R: Summary}
We show that even the familiar (and widely popular) field of complex numbers $\mathbb C$ contains non-zero infinitesimals. We try to convince the reader that infinitesimals are not only unavoidable in modern mathematics, but they are, actually,  very useful for simplifying mathematical reasoning by reducing the numbers of quantifiers compared with those in standard mathematics; just like infinitesimals were spectacularly successful (and perhaps, unavoidable) in creation of calculus long time ago. The simplicity of non-standard methods is especially noticeable in the presence of plenty of infinitesimals - the so called \emph{saturation}.
\end{remark}
\noindent \textbf{Acknowledgement}: The author is grateful to Ivan Penkov for the numerous fruitful discussions on the topic during the preparation of the manuscript.\\

\noindent\textbf{Declaration of competing interest:} The author declares that he has no known competing financial interests or personal relationships that could influence the work on this article. \\
\noindent\textbf{Declaration of funding:} The author declares that the work on this article has not been supported financially by any organization, including grants. No funding was received. 

\end{document}